\documentclass[reqno,tbtags,a4paper,12pt]{amsart}

\usepackage[in]{fullpage}

\usepackage{amsmath,amssymb,amsthm,amsfonts,amscd,graphicx}
\usepackage{mathrsfs}

\newcommand{\CAT}{C\!AT}
\newcommand{\bS}{\mathbb{S}}
\newcommand{\R}{\mathbb{R}}

\newcommand{\hM}{\widehat{M}}

\newcommand{\Z}{\mathbb{Z}}

\newcommand{\Aut}{\mathop{\mathrm{Aut}}\nolimits}

\newcommand{\T}{\mathcal{T}}

\newcommand{\U}{\mathcal{U}}

\newcommand{\Lob}{\mathbb{H}}

\newcommand{\Isom}{\mathop{\mathrm{Isom}}\nolimits}
\newcommand{\CC}{\mathscr{C}}
\newcommand{\RZ}{\mathcal{R}}
\newcommand{\dist}{\mathop{\mathrm{dist}}\nolimits}
\newcommand{\F}{\mathbf{F}}
\newcommand{\diam}{\mathop{\mathrm{diam}}\nolimits}
\newcommand{\cone}{\mathop{\mathrm{cone}}\nolimits}
\newcommand{\arcsinh}{\mathop{\mathrm{arsinh}}\nolimits}
\newcommand{\M}{\mathscr{M}}

\newtheorem{theorem}{Theorem}[section]
\newtheorem{propos}[theorem]{Proposition}
\newtheorem{cor}[theorem]{Corollary}

\newtheorem{conj}[theorem]{Conjecture}
\newtheorem{quest}[theorem]{Question}
\newtheorem*{URC}{Property URC}
\theoremstyle{definition}

\newtheorem{remark}[theorem]{Remark}

\author{Alexander A. Gaifullin}

\thanks{The work was partially supported by RFBR  (project 11-01-00694), by a grant of the President of Russian Federation (project MD-4458.2012.1), by a grant of the Government of the Russian Federation (project 2010-220-01-077), and by a programme of the Branch of Mathematical Sciences of the Russian Academy of Sciences.}

\title{Universal realisators for homology classes}

\date{}

\address{Steklov Mathematical Institute, Moscow\newline ${}$\quad
Moscow State University\newline ${}$\quad Institute for Information Transmission Problems, Moscow}

\email{agaif@mi.ras.ru}

\begin{document}

\begin{abstract}
We study oriented closed manifolds~$M^n$ possessing the following \textit{Universal Realisation of Cycles \textnormal{(}URC\textnormal{)} Property\/}: For each topological space~$X$ and each homology class $z\in H_n(X,\Z)$, there exist a finite-sheeted covering $\hM^n\to M^n$ and a continuous mapping $f:\hM^n\to X$ such that $f_*[\hM^n]=kz$ for a non-zero integer~$k$. We find wide class of examples of such manifolds~$M^n$ among so-called small covers of simple polytopes. In particular, we find $4$-dimensional hyperbolic manifolds possessing the URC property. As a consequence, we obtain that for each $4$-dimensional oriented closed manifold~$N^4$, there exists a mapping of non-zero degree of a hyperbolic manifold~$M^4$ to~$N^4$.  This was conjectured by Kotschick and L\"oh.  
\end{abstract}

\maketitle

\section{Introduction}
\label{section_intro}

Unless otherwise stated, all manifolds are supposed to be connected, compact, without boundary, and oriented.

Let $M^n$ and $N^n$ be manifolds of the same dimension~$n$. We say that $M^n$ \textit{dominates\/}~$N^n$ and write $M^n\geqslant N^n$ if there exists a non-zero degree mapping $M^n\to N^n$. Domination is a transitive relation on the set of homotopy types of connected oriented closed manifolds. This relation goes back to  the works of Milnor and Thurston~\cite{MiTh77} and Gromov~\cite{Gro82} and was first explicitly defined in the paper of Carlson and Toledo~\cite{CaTo89} with a reference to a lecture of Gromov. Good surveys of results on domination relation can be found in~\cite{Wan02} for the $3$-dimensional case and in~\cite{KoLo08} for an arbitrary dimension. Obviously, $M^n\geqslant S^n$ for every $n$-dimensional manifold~$M^n$. Hence the sphere~$S^n$ is the minimal element with respect to the domination relation.   
We shall focus on the following question of Carlson and Toledo~\cite{CaTo89} and the conjectural answer to it of Kotschick and L\"oh~\cite{KoLo08}.

\begin{quest}[Carlson--Toledo, 1989]
\label{quest_main} 
Is there an easily describable maximal class of homotopy types with respect to the domination relation, that is, a collection~$\CC_n$ of $n$-dimensional manifolds such that given any manifold~$N^n$, there exists an $M^n\in\CC_n$ satisfying $M^n\geqslant N^n$?
\end{quest}

Recall that a \textit{hyperbolic manifold\/} is a manifold admitting a Riemannian metric of constant negative sectional curvature.

\begin{conj}[Kotschick--L\"oh, 2008]\label{conj_main}
In every dimension $n\ge 2$, closed oriented hyperbolic manifolds represent a maximal class of homotopy types with respect to the domination relation.
\end{conj}

This conjecture is trivially true for $n=2$. By a result of Brooks~\cite{Bro85}, it is also true for $n=3$. Besides, Kotschick and L\"oh~\cite{KoLo08} showed that it is true for $4$-dimensional manifolds with finite fundamental groups. One of the goals of the present paper is to prove Conjecture~\ref{conj_main} for $n=4$.

The following answer to Question~\ref{quest_main} was obtained by the author~\cite{Gai08a},~\cite{Gai08b} in 2008. In every dimension~$n$, there is a manifold~$M^n_0$ such that each manifold~$N^n$ is dominated by a finite-sheeted covering $\hM^n_0$ of $M^n_0$. Thus we can take for~$\CC_n$ the collection of all finite-sheeted coverings of~$M^n_0$. The manifold~$M^n_0$ is a so-called \textit{small cover\/} of the permutahedron (see sections~\ref{section_sc} and~\ref{section_urc_sc} for definitions). By a result of Tomei~\cite{Tom84}, $M^n_0$ is homeomorphic to an isospectral manifold of $(n+1)\times(n+1)$ tridiagonal symmetric real matrices, which is important for integrable systems, since it carries the famous Toda flow. 
However, for $n\ge 3$, the manifold~$M^n_0$ is not hyperbolic, since $\pi_1(M^n_0)$ contains a free Abelian subgroup of rank~$2$. We say that $M^n$ \textit{virtually dominates\/}~$N^n$ if there exists a finite-sheeted covering~$\hM^n$ of~$M^n$ such that $\hM^n\geqslant N^n$. Then the author's result in~\cite{Gai08a},~\cite{Gai08b}  can be formulated as follows.

\begin{theorem}\label{theorem_M0}
The manifold~$M^n_0$ is maximal with respect to the virtual domination relation, i.\,e., it virtually dominates every $n$-dimensional manifold~$N^n$.
\end{theorem}

Gromov~\cite{Gro87} suggested a \textit{hyperbolisation procedure\/} that for each polyhedron~$P$, yields a polyhedron~$P_h$ and a mapping~$f:P_h\to P$ such that $P_h$ has a metric of non-positive curvature in sense of Alexandrov (see definition in section~\ref{section_urc_sc}), $f$ induces an injection in integral cohomology, and singularities of~$P_h$ are ``not worse'' than  singularities of~$P$. In particular, if~$P$ is a PL manifold, then $P_h$ is also a PL manifold and $f:P_h\to P$ is a degree~$1$ mapping. Besides,  $f$ takes the rational Pontryagin classes of~$P$ to the rational Pontryagin classes of~$P_h$. This result of Gromov implies that each manifold can be dominated (with degree~$1$) by a manifold of non-positive curvature. Gromov's hyperbolisation procedure has been improved by Davis--Januszkiewicz~\cite{DaJa91b} and Charney--Davis~\cite{ChDa95}. The strongest result has recently been  obtained by Ontaneda~\cite{Ont11}, who has proved that, for a smooth manifold~$P$, $P_h$ can be constructed to be a Riemannian manifold of negative sectional curvature in an arbitrarily small interval~$[-1-\varepsilon,-1]$. However, even the class of all Riemannian manifolds with such pinched sectional curvature seems to be too wide for Question~\ref{quest_main}. Theorem~\ref{theorem_M0} shows that more narrow classes~$\CC_n$ can be found.  

Until now most results on domination relation in dimensions $n>3$ have been negative. For example, it is known that there exist manifolds that cannot be dominated by a K\"ahler manifold~\cite{CaTo89} and there exist manifolds that cannot be dominated by a direct product $M_1\times M_2$ of two manifolds of positive dimensions~\cite{KoLo08}.

The domination relation is closely related to Steenrod's problem on realisation of cycles. The classical question of Steenrod is as follows. Given a homology class $z\in H_n(X,\Z)$ of a topological space~$X$, does there exist a smooth manifold $M^n$ and a continuous mapping~$f:M^n\to X$ such that $f_*[M^n]=z$? If such $M^n$ and~$f$ exist, we shall say that $z$ is \textit{realisable\/}. A well-known theorem of Thom~\cite{Tho54} claims that there exist non-realisable homology classes, but, for each~$z$,  a certain multiple~$kz$, $k\in\Z_{>0}$, is realisable. An important question is to describe a class~$\CC_n$ of $n$-dimensional manifolds such that each $n$-dimensional homology class of each space~$X$ can be realised as an image of the fundamental class of a manifold belonging to~$\CC_n$.  Thom's theorem easily implies that this question is equivalent to Question~1.1. In particular, a manifold~$M^n$ is  maximal with respect to the virtual domination relation if and only if it has the following \textit{Universal Realisation of Cycles\/} (URC) \textit{Property\/}.

\begin{URC}
For each~$X$ and each $z\in H_n(X,\Z)$, there exist a finite-sheeted covering~$\hM^n$ of~$M^n$ and a mapping $f:\hM^n\to X$ such that $f_*\bigl[\hM^n\bigr]=kz$ for a non-zero integer~$k$.
\end{URC}

Manifolds satisfying this condition will be called \textit{URC-manifolds}. By Theorem~\ref{theorem_M0}, URC-manifolds exist in every dimension. 
It is an interesting problem to describe the class of URC-manifolds. The following proposition is straightforward.

\begin{propos}\label{propos_obvious}
1) Suppose $\hM^n$ is a finite-sheeted covering of~$M^n$; then $\hM^n$ is a URC-manifold if and only if $M^n$ is a URC-manifold.

2) Suppose $M^n$ is a URC-manifold; then the connected sum $M^n\# N^n$ is a URC-manifold for every~$N^n$. 

3) Suppose $N^n$ is a URC-manifold; then $M^n$ is a URC-manifold if and only if $M^n$ virtually dominates~$N^n$.
\end{propos}

In this paper, our goal is to obtain a wide collection of examples of URC-manifolds. 
These examples are small covers of simple polytopes satisfying certain special conditions. In particular, in dimensions $n\le 4$ we shall find hyperbolic URC-manifolds. Thus we shall prove  Conjecture~\ref{conj_main} for $n=4$. 

We denote by $\bS^n$, $\R^n$, and $\Lob^n$ the spaces of constant curvature~$+1$, $0$, and~$-1$ respectively, i.\,e., the round sphere of radius~$1$, the Euclidean space, and the Lobachevsky space.

Each compact hyperbolic manifold $M^n$ is isometric to the quotient of~$\Lob^n$ by the action of a torsion-free uniform discrete transformation group~$\Gamma\subset\Isom(\Lob^n)$.    
Recall that two discrete subgroups $\Gamma_1,\Gamma_2\subset\Isom(\Lob^n)$ are said to be \textit{commensurable\/} if $\Gamma_1\cap\Gamma_2$ has finite indices in both~$\Gamma_1$ and~$\Gamma_2$.
A discrete subgroup $W\subset\Isom(\Lob^n)$ is called a \textit{reflection subgroup\/} if it is generated by orthogonal reflections in hyperplanes $H\subset\Lob^n$. A reflection subgroup is called \textit{right-angular\/} if the mirrors of any two reflections $s,s'\in W$ either are orthogonal to each other or do not intersect.

\begin{theorem}\label{theorem_hyp}
Let $M^n=\Lob^n/\Gamma$ be a compact hyperbolic manifold such that the group $\pi_1(M^n)=\Gamma$ is commensurable with a uniform right-angular reflection subgroup $W\subset\Isom(\Lob^n)$. Then  $M^n$ is a URC-manifold.	
\end{theorem}

Uniform right-angular reflection subgroups $W\subset\Isom(\Lob^n)$ correspond to compact right-angular polytopes $P\subset\Lob^n$. (It is well known that every compact right-angular polytope is simple.) Compact right-angular polytopes~$P\subset\Lob^n$ exist for $n=2,3,4$.  For example, there exist the  right-angular regular dodecahedron in~$\Lob^3$ and the right-angular regular $120$-cell in~$\Lob^4$. Many examples of compact right-angular polyhedra in~$\Lob^3$ were obtained by L\"obell~\cite{Loe31}.
Hence, for $n=2,3,4$, there exist uniform right-angular reflection subgroups $W\subset\Isom(\Lob^n)$.
By Selberg's lemma each such subgroup~$W$ contains a finite index torsion-free subgroup~$\Gamma$. Then~$\Lob^n/\Gamma$ is a  hyperbolic URC manifold.

\begin{cor}
In dimensions $n=2,3,4$, there exist hyperbolic URC-manifolds. Consequently, each manifold~$N^n$ can be dominated by a hyperbolic manifold.
\end{cor}

Notice that even in dimension~$3$ this assertion is stronger than a result of Brooks~\cite{Bro85} because the hyperbolic manifolds obtained by the Brooks construction are not coverings of the same  manifold. 

By a result of Vinberg~\cite{Vin84}, for $n\ge 5$, there are no compact right-angular polytopes $P\subset\Lob^n$ and, hence, there are no uniform right-angular reflection subgroups $W\subset\Isom(\Lob^n)$. Thus Theorem~\ref{theorem_hyp} cannot be used for $n\ge 5$. 

This paper is organised in the following way. 
In section~\ref{section_sc} we describe several constructions of manifolds glued of simple polytopes, in particular, the construction of so-called \textit{small covers\/} of simple polytopes due to Davis and Januszkiewicz~\cite{DaJa91a}. 
These constructions play key role throughout the paper. Main results are formulated in section~\ref{section_urc_sc}. They include several sufficient conditions for small covers of simple polytopes to be URC-manifolds. In particular, we prove that a small cover of a polytope is a URC-manifold whenever it admits an equivariant metric of strictly negative curvature in sense of Alexandrov (with at least one ``smooth'' point). The proofs of these results as well as the proof of Theorem~\ref{theorem_hyp} are given in sections~\ref{section_proof_sc} and~\ref{section_proof_hyp}.

Sections~\ref{section_group_constr} and~\ref{section_Pi_constr} are devoted to the proof of Theorem~\ref{theorem_M0}. Actually, this theorem has been proved by the author in~\cite{Gai08a},~\cite{Gai08b}. However, in the present paper we give a new interpretation of this proof based on a certain group-theoretic construction that essentially clarifies the proof.

The author is grateful to V.~M.~Buchstaber and E.~B.~Vinberg for useful comments.

\section{Manifolds glued of simple polytopes}
\label{section_sc}

We denote by~$[m]$ the set $\{1,2,\ldots,m\}$.

Let $S$ be a finite set and let $G$ be a graph on the vertex set $S$. Recall that a \textit{right-angular Coxeter group} associated to $G$ is the group~$W$ with the set of generators $S$ such that all relations in $W$ are consequences of relations $s^2=1$, $s\in S$, and $st=ts$ whenever $\{s,t\}$ is an edge of~$G$. The pair $(W,S)$ is called a \textit{right-angular Coxeter system\/}.

The following general construction is due to Vinberg~\cite{Vin71}. Let $X$ be a topological space 
and let $X_s$ be closed subsets of it indexed by elements $s\in S$. For a point $p\in X$, denote by $W(p)$ the subgroup of $W$ generated by all elements $s\in S$ such that $p\in X_s$. Endow $W$ with the discrete topology. Consider an equivalence relation~$\sim$ on $X\times W$ such that $(p,g)\sim (p',g')$ if and only if $p=p'$ and $g'g^{-1}\in W(p)$, and denote the quotient space $(X\times W)/\sim$ by $\U(W,X,\{X_s\})$. Denote by $[p,g]$ the point of $\U(W,X,\{X_s\})$ corresponding to the equivalence class of~$(p,g)$. The natural \textit{right} action of $W$ on $X\times W$ is compatible with~$\sim$. Hence the right action of $W$ on $\U(W,X, \{X_s\})$ given by $[p,g]\cdot g'=[p,gg']$ is well defined.

\begin{remark}
Usually in this construction the condition $g'g^{-1}\in W(p)$ is replaced with the condition $g^{-1}g'\in W(p)$ to obtain a \textit{left\/} action of~$W$ on $\U(W,X, \{X_s\})$. However, it is convenient for us to have a \textit{right\/} action here.
\end{remark}

Davis~\cite{Dav83}, \cite{Dav87}, and Davis and Januszkiewicz~\cite{DaJa91a} applied this construction to obtain a wide class of manifolds glued of simple convex polytopes. Here we describe several their constructions. 

Let $P$ be a simple convex polytope in either~$\R^n$ or~$\Lob^n$ with facets $F_1,\ldots,F_m$. We always assume that $P$ is $n$-dimensional, i.\,e., is not contained in a hyperplane, and compact. Recall that a polytope is said to be \textit{simple\/} if every its vertex is contained in exactly $n$ facets.

1. Let $S=\{s_1,\ldots,s_m\}$ and let $\{s_i,s_j\}$ be an edge of~$G$ if and only if $F_i\cap F_j\ne\emptyset$. Denote by $W_P$ the right-angular Coxeter group associated to the graph~$G$. Let $X=P$ and $X_{s_i}=F_i$, $i=1,\ldots,m$. Put $\U_P=\U(W_P,P,\{F_i\})$. Then
$$
\U_P=(P\times W_P)/\sim,
$$
where $(p,g)\sim (p',g')$ if and only if $p=p'$ and $g'g^{-1}$ belongs to the subgroup $W(p)\subset W_P$ generated by all~$s_i$ such that~$p\in F_i$.

2. Let $S$, $X$, $X_{s_i}$ be as in the previous construction and let $G$ be a complete graph on the vertex set $S$. Then the corresponding right-angular Coxeter group is~$\Z_2^m$. We shall use the multiplicative notation for~$\Z_2^m$ and we shall denote its generators by $a_1,\ldots,a_m$ to avoid confusion with the generators $s_1,\ldots,s_m$ of~$W_P$. Put $\RZ_P=\U(\Z_2^m,P,\{F_i\})$. Then
$$
\RZ_P=(P\times\Z_2^m)/\sim,
$$
where $(p,g)\sim (p',g')$ if and only if $p=p'$ and $g'g^{-1}$ belongs to the subgroup of~$\Z_2^m$ generated by all~$a_i$ such that~$x\in F_i$. 

3. Suppose that facets of~$P$ are coloured in colours from a finite set. We say that a colouring is \textit{regular\/} if any two intersecting facets are of distinct colours. Assume that~$P$ admits a regular colouring of facets in $n$ colours, which we denote by $1,\ldots,n$. Let $S=\{b_1,\ldots,b_n\}$, $X=P$, and let $X_{b_i}$ be the union of all facets of colour $i$. Let $G$ be a complete graph on the vertex set $S$. Then the corresponding right-angular Coxeter group is~$\Z_2^n$.  Put $M_P=\U(\Z_2^n,P,\{X_{b_i}\})$. Then 
$$
M_P=(P\times\Z_2^n)/\sim,
$$
where $(p,g)\sim (p',g')$ if and only if $p=p'$ and $g'g^{-1}$ belongs to the subgroup of~$\Z_2^n$ generated by all~$b_i$ such that~$p$ is contained in a facet of colour~$i$. The manifold $M_P$ is called the \textit{small cover of~$P$ induced from a linear model\/}. Notice that if a regular colouring of facets of~$P$ exists, it is unique up to a permutation of colours.   

4. This is a generalization of the previous construction. A \textit{characteristic function\/} for~$P$ is a mapping $\lambda:[m]\to\Z_2^n$ that satisfies the following condition:
\begin{itemize}
\item[] \textit{The elements $\lambda(i_1),\ldots,\lambda(i_n)$  generate~$\Z_2^n$ whenever $i_1,\ldots,i_n$ are pairwise distinct and $F_{i_1}\cap\ldots\cap F_{i_n}\ne\emptyset$.}
\end{itemize}  
Given a characteristic function, we can define a \textit{small cover\/} of~$P$ by
$$
M_{P,\lambda}=(P\times\Z_2^n)/\sim_{\lambda},
$$
where $(p,g)\sim_{\lambda} (p',g')$ if and only if $p=p'$ and $g'g^{-1}$ belongs to the subgroup of~$\Z_2^n$ generated by all~$\lambda(a_i)$ such that~$p\in F_i$. If every $\lambda(i)$ is contained in the basis $b_1,\ldots,b_n$, we obtain the small covering induced from a linear model described above.
Notice that there exist polytopes that admit no characteristic functions~$\lambda$.

\begin{remark}
The manifolds~$\RZ_P$ and~$M_{P,\lambda}$ have \textit{``complex''\/} analogues that are obtained by replacing the groups~$\Z_2^m$ and~$\Z_2^n$ with the compact tori~$T^m$ and~$T^n$ and are called \textit{moment-angle manifolds\/} and \textit{quasi-toric manifolds\/} respectively. These manifolds have also been defined by Davis and Januszkiewicz~\cite{DaJa91a}. The theory of moment-angle manifolds and quasi-toric manifolds has been fruitfully developed and has found many applications, see~\cite{BuPa02} and references therein. 
\end{remark}

The spaces $\U_P$, $\RZ_P$, and $M_{P,\lambda}$ are glued of copies of the polytope~$P$ indexed by elements of the groups~$W_P$, $\Z_2^m$, and~$\Z_2^n$ respectively. Hence these spaces have natural cell decompositions with each $n$-dimensional cell isomorphic to~$P$. 
It is not hard to see that each vertex $v$ of each of these decompositions is contained in exactly $2^n$ \ $n$-dimensional cells. Moreover, the corners of these $2^n$ cells meet at vertex~$v$ in a standard way, i.\,e., as the corners of orthants at the origin of~$\R^n$.  
Hence $\U_P$, $\RZ_P$, and $M_{P,\lambda}$ are PL manifolds. Besides, the dual cell decompositions~$\U_P^*$, $\RZ_P^*$, and $M_{P,\lambda}^*$  are decompositions into cubes.

The manifolds~$\U_P$, $\RZ_P$, and $M_{P,\lambda}$ carry natural right actions of the groups $W_P$, $\Z_2^m$, and~$\Z_2^n$ respectively and
$$
\U_P/W_P=\RZ_P/\Z_2^m=M_{P,\lambda}/\Z_2^n=P.
$$ 
Further, we have surjective homomorphisms $\eta:W_P\to \Z_2^m$ and $\lambda_*:\Z_2^m\to\Z_2^n$ given by $\eta(s_i)=a_i$ and $\lambda_*(a_i)=\lambda(i)$ respectively. It can be easily seen that the groups $\ker\eta$ and $\ker\lambda_*$ act freely on~$\U_P$ and~$\RZ_P$ respectively and
$$
\U_P/\ker\eta=\RZ_P,\qquad \RZ_P/\ker\lambda_*=M_{P,\lambda}.
$$
Thus $\U_P$ is a regular covering of~$\RZ_P$ and~$\RZ_P$ is a $2^{m-n}$-sheeted regular covering of~$M_{P,\lambda}$.
Vice versa, the quotient of~$\RZ_P$ by any subgroup~$A\subset\Z_2^m$, $A\cong\Z_2^{m-n}$, that acts freely on it is a small cover of~$P$. 

Davis~\cite{Dav83} proved that the manifold~$\U_P$ is simply connected. Hence $\U_P$ is the universal covering of the manifolds~$\RZ_P$ and~$M_{P,\lambda}$.
A simple polytope~$P$ is called a \textit{flag polytope\/} if a collection $F_{i_1},\ldots,F_{i_r}$ of its facets has non-empty intersection whenever all pairwise intersections $F_{i_k}\cap F_{i_l}$ are non-empty. Davis~\cite{Dav83} proved that if $P$ is a flag polytope, then the manifold~$\U_P$ is contractible. Hence, the manifolds~$\RZ_P$ and~$M_{P,\lambda}$ are aspherical and their fundamental groups are isomorphic to the groups~$\ker\eta$ and~$\ker(\lambda_*\eta)$ respectively.

To a simple polytope~$P$ with facets $F_1,\ldots,F_m$ is assigned a simplicial complex~$K_P$ on the vertex set $\{v_1,\ldots,v_m\}$ such that vertices $v_{i_1},\ldots,v_{i_k}$ span a simplex in~$K_P$ if and only if $F_{i_1}\cap\ldots\cap F_{i_k}\ne\emptyset$. The complex~$K_P$ is an $(n-1)$-dimensional combinatorial sphere. (A simplicial complex~$K$ is called an $(n-1)$-dimensional \textit{combinatorial sphere\/} if it is PL homeomorphic to the boundary of the $n$-dimensional simplex.) If $P$ is a simple polytope in~$\R^n$, then $K_P$ is isomorphic to the boundary of the dual simplicial polytope~$P^*$. 

The constructions of manifolds~$\U_P$, $\RZ_P$, and~$M_{P,\lambda}$ can be extended to a wider class of objects than simple polytopes.
It is well known that not every combinatorial sphere can be realised as the boundary of a simplicial polytope. Nevertheless, it is convenient to assign to every combinatorial sphere~$K$ certain object~$P_K$ which turns to be the dual simple polytope whenever $K$ is the boundary of a simplicial polytope. The construction is as follows. Suppose $K$ is an $(n-1)$-dimensional combinatorial sphere on the vertex set~$\{v_1,\ldots,v_m\}$.
Let $P_K=\cone(K')$ be the cone over the barycentric subdivision of~$K$ and let $F_i$ be the star of a vertex~$v_i$ in~$K'$, that is, a subcomplex of~$K'$ consisting of all simplices containing~$v_i$. We  say that~$P_K$ is the \textit{simple cell\/} dual to~$K$, $F_i$ are \textit{facets\/} of~$P_K$, and non-empty intersections of facets are \textit{faces\/} of~$P_K$. The simple cell~$P_K$ has the following properties that mimic properties of simple polytopes: 
\begin{itemize}
\item $P_K$ is PL homeomorphic to the $n$-dimensional simplex;
\item $\partial P_K=F_1\cup\ldots\cup F_m$; 
\item $F_{i_1}\cap\ldots\cap F_{i_k}$ is PL homeomorphic to the $(n-k)$-dimensional simplex if the  vertices $v_{i_1},\ldots, v_{i_k}$ span a $(k-1)$-dimensional simplex of~$K$, and $F_{i_1}\cap\ldots\cap F_{i_k}$ is empty if the vertices $v_{i_1},\ldots, v_{i_k}$ do not span a simplex of~$K$. 
\end{itemize}
The described constructions $P\mapsto K_P$ and $K\mapsto P_K$ are mutually inverse and yield a one-to-one correspondence between $n$-dimensional simple cells and $(n-1)$-dimensional combinatorial spheres. 
The constructions of the manifolds~$\U_P$, $\RZ_P$, and~$M_{P,\lambda}$ are immediately extended to the case of an arbitrary simple cell~$P$.

\begin{remark}\label{remark_pol_comp}
Indeed, the requirement that $K$ is a combinatorial sphere is not needed to construct~$P_K$. This construction can be applied to an arbitrary simplicial complex~$K$, and $P_K$ is called a \textit{simple polyhedral complex\/} dual to~$K$ (see~\cite{Dav87},~\cite{DaJa91a}). The spaces~$\U_{P_K}$, $\RZ_{P_K}$, and~$M_{P_K,\lambda}$ can be constructed for an arbitrary simple polyhedral complex~$P$. However, generally they are not manifolds.  Still, they are homology manifolds if $K$ is a homology sphere, and they are pseudo-manifolds if~$K$ is a pseudo-manifold. 
\end{remark}

Let us fix terminology and notation concerning simplicial complexes and simple cells. First, we always work with \textit{geometric\/} simplcial complexes. Second, if~$K$ is a simplicial complex, we use the same notation~$K$ (instead of~$|K|$) for the underlying topological space of~$K$. This causes no ambiguity unless we write~$x\in K$. We make a convention that notation~$x\in K$ always means that $x$ is a point of the underlying space of~$K$ and we never use notation~$x\in K$ to indicate that $x$ is a simplex of~$K$. The simplex with vertices $u_1,\ldots,u_k$ will be denoted by $[u_1\ldots u_k]$. If $K$ is a combinatorial sphere and $P$ is the dual simple cell, we always identify with each other the  topological spaces $\partial P$, $K'$, and~$K$. Now, let $K_1$ and $K_2$ be combinatorial spheres and let $P_1$ and $P_2$ be the dual simple cells respectively. A simplicial mapping $\varphi:K_1\to K_2$ induces the simplicial mapping $\varphi':K_1'\to K_2'$ and the mapping $\bar{\varphi}=\cone(\varphi'):P_1\to P_2$.
After identifying $\partial P_i$ with $K_i$, $i=1,2$, the mapping $\bar{\varphi}|_{\partial P_1}$ will coincide with the original mapping $\varphi:K_1\to K_2$. Let $u$ and $v$ be vertices of~$K_1$ and $K_2$ respectively and let $F$ and $G$ be the corresponding facets of~$P_1$ and $P_2$ respectively. It is easy to see that $\bar{\varphi}(F)\subset G$ whenever $\varphi(u)=v$.

Let $P_1$ and $P_2$ be simple cells of the same dimension and let $f:P_1\to P_2$ be a mapping such that $f(\partial P_1)\subset \partial P_2$. The  \textit{degree\/} of~$f$ is, by definition, the degree of the induced mapping $P_1/\partial P_1\to P_2/\partial P_2$.

\section{URC small covers}
\label{section_urc_sc}

By definition, the \textit{permutahedron\/} $\Pi^n\subset \R^{n+1}$ is the convex hull of $(n+1)!$ points $(j_1,j_2,\ldots,j_{n+1})$ such that $j_1,j_2,\ldots,j_{n+1}$ is a permutation of the set $1,2,\ldots,n+1$. It can be easily seen that $\Pi^n$ is a simple polytope in the hyperplane $\sum_{i=1}^{n+1}t_i=\frac{(n+1)(n+2)}{2}$, where $t_1,\ldots,t_n$ are the standard coordinates in~$\R^{n+1}$. The permutahedron $\Pi^n$ has $2^{n+1}-2$ facets~$F_{\omega}$, which are conveniently indexed by non-empty proper subsets $\omega\subset[n+1]$. (A subset $\omega\subset[n+1]$ is called \textit{proper\/} if $\omega\ne[n+1]$.) The facet $F_{\omega}$ is given by the equation
$$
\sum_{i\in\omega}t_i=\sum_{j=n-|\omega|+2}^{n+1}j=\frac{(n+1)(n+2)}{2}-\frac{(n-|\omega|+1)(n-|\omega|+2)}{2}
$$
where $|\omega|$ is the cardinality of $\omega$. It is easy to see that facets $F_{\omega_1}$ and~$F_{\omega_2}$ have non-empty intersection if and only if either $\omega_1\subset\omega_2$ or $\omega_2\subset\omega_1$. This implies that the simplicial complex~$K_{\Pi^n}$ is isomorphic to the barycentric subdivision of the boundary of the $n$-simplex.

Colouring each facet~$F_{\omega}$ in colour~$|\omega|$, we obtain a regular colouring of facets of~$\Pi^n$ in $n$ colours $1,\ldots,n$. Hence the small cover~$M_{\Pi^n}$ induced from a linear model is well defined. The following theorem is a reformulation of Theorem~\ref{theorem_M0}.

\begin{theorem}\label{theorem_MPi}
The manifold~$M_{\Pi^n}$ is a URC-manifold.
\end{theorem} 

This theorem was proved by the author~\cite{Gai08a},~\cite{Gai08b} (see also~\cite{Gai08c} for further applications of this construction). In sections~\ref{section_group_constr} and~\ref{section_Pi_constr} we shall give a proof of Theorem~\ref{theorem_MPi}. Actually, this is the same proof as the proof given in~\cite{Gai08a} and~\cite{Gai08b}. Nevertheless, in the present paper we give a new exposition of this proof using a new group-theoretic lemma, which essentially clarifies the idea behind the construction.   

For each characteristic function~$\lambda$, $\RZ_P$ is a finite-sheeted covering of~$M_{P,\lambda}$. Hence $M_{P,\lambda}$ is a URC-manifold if and only if $\RZ_P$ is a URC-manifold. In particular, we see that $\RZ_{\Pi^n}$ is a URC-manifold and $M_{\Pi^n,\lambda}$ is a URC-manifold for every characteristic function~$\lambda$. We shall give several sufficient conditions for a manifold~$\RZ_P$ to be a URC-manifold. To formulate the result we need several definitions. 

A simplicial complex~$K$ is said to be a \textit{flag complex\/} if a set of vertices of~$K$ spans a simplex whenever these vertices are pairwise joined by edges. Another  is that $K$ contains ``no empty simplices'' or satisfies ``no-$\triangle$-condition'' (see~\cite{Gro87}). 
Further, we shall say that $K$ does not contain an \textit{empty $4$-circuit\/} if for any pairwise distinct vertices $u_1,u_2,u_3,u_4$ such that  $[u_1u_2]$, $[u_2u_3]$, $[u_3u_4]$, and $[u_4u_1]$ are edges of~$K$, at least one of the diagonals~$[u_1u_3]$ and~$[u_2u_4]$ is an edge of~$K$. In Gromov's terminology~\cite{Gro87} this condition is called ``Siebenmann's no-$\square$-condition''.

Let $K$ be an $(n-1)$-dimensional combinatorial sphere and let $0<\varepsilon<\pi$. We say that $K$ is \textit{$\varepsilon$-fine\/} if there exists a mapping $K\to \bS^{n-1}$ of non-zero degree such that the diameter of the image of every simplex of~$K$ is less than~$\varepsilon$. A simple cell $P$ is said to be \textit{$\varepsilon$-fine\/} if $K_P$ is $\varepsilon$-fine.

\begin{theorem}\label{theorem_sc}
Let $P$ be an $n$-dimensional simple cell. Assume that $P$ satisfies one of the following conditions:
\begin{itemize}
\item[(a)] There exists a simplicial mapping $f:K_P\to (\partial \Delta^n)'$ of non-zero degree, where $(\partial \Delta^n)'$ is the barycentric subdivision of the boundary of the $n$-dimensional simplex.
\item[(b)] $K_P$ is isomorphic to the barycentric subdivision of a combinatorial sphere~$L$.
\item[(c)] $P$ is $\varepsilon_n$-fine, where \begin{equation*}
\varepsilon_n=\arccos\left(1-\frac{12}{n(n+1)(n+2)}\right).
\end{equation*}
\item[(d)] $P\subset \Lob^n$ is a simple convex polytope whose interior contains a closed ball of radius 
\begin{equation*}
\rho_n=
\log\left(\sqrt{\frac{n(n+1)(n+2)}{6}}+
\sqrt{\frac{n(n+1)(n+2)}{6}-1}\right).
\end{equation*} 
\item[(e)] $K_P$ is a flag combinatorial sphere without an empty $4$-cir\-cuit, i.\,e., $K_P$ satisfies the ``no-$\triangle$-condition'' and the ``no-$\square$-condition''.
\end{itemize}
Then $\RZ_P$ is a URC-manifold. Therefore, $M_{P,\lambda}$ is a URC-manifold for every characteristic function~$\lambda$.  
\end{theorem}

The assertion that each manifold~$\RZ_P$ satisfying condition~(x) is a URC-manifold will be refered as Theorem~\ref{theorem_sc}(x), where x is one of the letters~a, b, c, d, and~e.
Conditions~(a), (b), and (e) are purely combinatorial. Notice, however, that condition~(a) is harder to check than conditions~(b) and~(e). Conditions~(c) and (d)  have more geometric nature. We shall see that (d)$\Rightarrow$(c)$\Rightarrow$(a) and (b)$\Rightarrow$(a). Hence all assertions of Theorem~\ref{theorem_sc} except for~(e) will follow from assertion~(a). 

The proof of Theorem~\ref{theorem_sc}(e) uses theory of $\CAT(\kappa)$ spaces. (Good references for this subject are~\cite{Gro87} and~\cite{BrHa99}; for applications to manifolds~$\U_P$, see~\cite{Dav02}.) Recall the definition of a $\CAT(-1)$ space. Let $X$ be a complete geodesic metric space. For every geodesic triangle~$xyz$ in $X$, a \textit{comparison triangle\/} is a triangle~$x^*y^*z^*$ in~$\Lob^2$ with the same edge lengths. Let $p$ and $q$ be arbitrary points in the boundary of~$xyz$ and let~$p^*$ and~$q^*$ be the corresponding points in the boundary of~$x^*y^*z^*$. The $\CAT(-1)$ \textit{inequality\/} is the inequality $\dist_{X}(p,q)\le\dist_{\Lob^2}(p^*,q^*)$. A space~$X$ is said to be a $\CAT(-1)$ \textit{space\/} if the $\CAT(-1)$ inequality is satisfied for every geodesic triangle~$xyz$ and every points~$p$ and~$q$ in its boundary. A space~$X$ is said to have \textit{curvature\/} $K\le -1$ if the $\CAT(-1)$ inequality is satisfied locally, i.\,e., in a neighbourhood of every point. Replacing in these definitions the Lobachevsky plane~$\Lob^2$ with the Euclidean plane~$\R^2$ one will obtain a definitions of a $\CAT(0)$ space and of a space of curvature $K\le 0$. 

\begin{theorem}\label{theorem_CAT(-1)}
Assume that the manifold~$\RZ_P$ admits a $\Z_2^m$-invariant metric of curvature $K\le -1$ such that there exists a point $o\in \RZ_P$ whose neighbourhood is isometric to an open subset of~$\Lob^n$. Then $\RZ_P$ is a URC-manifold.
\end{theorem}

The universal covering of a space of curvature $K\le -1$ is a $\CAT(-1)$ space. Hence  $\RZ_P$ has a $\Z_2^m$-invariant metric of curvature $K\le -1$ if and only if $\U_P$ has a $W_P$-invariant $\CAT(-1)$ metric. The condition of the existence of a point~$o$ with a ``standard'' neighbourhood is technical and probably can be avoided.

The standard way to construct a metric on~$\U_P$ is as follows. The manifold $\U_P$ has a natural cell decomposition with $n$-dimensional cells isomorphic to $P$. The dual cell decomposition~$\U_P^*$ is a cubical decomposition (see section~\ref{section_sc}). Now, choose an $a>0$ and endow every cube of~$\U_P^*$ with the metric of a regular cube in~$\Lob^n$ of edge length~$a$. Gromov~\cite{Gro87} proved that the obtained metric is $\CAT(-1)$ for some $a>0$ if and only if a simple cell~$P$ satisfies condition~(e) of Theorem~\ref{theorem_sc}, that is, satisfies the ``no-$\triangle$-condition'' and the ``no-$\square$-condition''. This $\CAT(-1)$ metric on~$\U_P$ is $W_P$-invariant. Besides, it is \textit{piecewise hyperbolic\/}, since every cell of~$\U_P^*$ is isometric to a cube in~$\Lob^n$. Hence Theorem~\ref{theorem_sc}(e)  immediately follows from Theorem~\ref{theorem_CAT(-1)}.

\begin{remark}
The condition $K\le -1$ in Theorem~\ref{theorem_CAT(-1)} cannot be replaced by a condition $K\le 0$. Gromov~\cite{Gro87} showed that the manifold~$\U_P$ has a $W_P$-invariant $\CAT(0)$ metric whenever~$P$ is a flag simple cell, i.\,e., whenever~$K_P$ satisfies the ``no-$\triangle$-condition''. For example, the direct product $P_1\times P_2$ of any flag simple cells is a flag simple cell. Hence the manifold $\RZ_{P_1\times P_2}$ has a $\Z_2^m$-invariant metric of curvature $K\le 0$. However, a result of Kotschick and L\"oh~\cite{KoLo08} implies that the manifold $\RZ_{P_1\times P_2}=\RZ_{P_1}\times\RZ_{P_2}$ cannot be a URC-manifold if both cells~$P_1$ and~$P_2$ have positive dimensions. Nevertheless, the following conjecture seems to be reasonable.    
\end{remark}

\begin{conj}
Suppose $P$ is a simple cell that is not combinatorially equivalent to a direct product of two simple cells of positive dimensions. Then $\RZ_P$ is a URC-manifold.
\end{conj}

Let us show that Theorem~\ref{theorem_CAT(-1)} implies Theorem~\ref{theorem_hyp}. Indeed, let $W\subset\Isom(\Lob^n)$ be a uniform right-angular reflection group and let $P$ be a fundamental domain for~$W$. Then $\U_P=\Lob^n$ has a $W$-invariant metric of constant curvature~$-1$. Theorem~\ref{theorem_CAT(-1)} yields that $\RZ_P$ is a URC-manifold. Now, for a torsion-free subgroup~$\Gamma\subset W$, the manifolds~$\RZ_P$ and~$\Lob^n/\Gamma$ possess a common finite-sheeted covering. Hence $\Lob^n/\Gamma$ is a URC-manifold. 

\begin{remark}
It is reasonable to ask whether the obtained URC-manifolds are smoothable. Indeed, if $P$ is not only a simple cell, but a simple polytope in either~$\R^n$ or~$\Lob^n$, then $P$ has a natural structure of smooth manifold with corners. Hence the manifold~$\RZ_P$ has a $\Z_2^m$-equivariant smooth structure, see~\cite{Dav83}.
\end{remark}

\begin{remark}
The proofs of Theorems~\ref{theorem_sc} and~\ref{theorem_CAT(-1)} that will be given in the next two sections will never use that $K_P$ is a combinatorial sphere. Actually, we might take for $P$ a  simple polyhedral complex dual to an arbitrary oriented pseudo-manifold~$K$ (see Remark~\ref{remark_pol_comp}). Then the analogues of Theorems~\ref{theorem_sc} and~\ref{theorem_CAT(-1)} would assert that the pseudo-manifold~$\RZ_P$ satisfies the URC-condition, though we would not be able to guarantee that $\RZ_P$ is a manifold.
\end{remark}

\section{A group-theoretic construction}
\label{section_group_constr}
Let $\Omega$ be a finite partially ordered set and let $m=|\Omega|$ be its cardinality. We shall write $\omega_1\nless\omega_2$ to indicate that either $\omega_1\ge\omega_2$ or $\omega_1$ and $\omega_2$ are incomparable. Let $\F$ be the free product of $m$ copies of the group~$\Z_2$
indexed by the elements of the set~$\Omega$. The generator of the factor~$\Z_2$ corresponding to an element~$\omega\in\Omega$ will be denoted by~$x_{\omega}$. Thus,
$$
\F=\langle x_{\omega},\omega\in\Omega\mid x_{\omega}^2=1\rangle.
$$
For each $\omega\in\Omega$, we consider the automorphism $\psi_{\omega}\in\Aut(\F)$ such that
$$
\psi_{\omega}(x_{\gamma})=\left\{
\begin{aligned}
x_{\omega}&x_{\gamma}x_{\omega}&&\text{if}\ \gamma<\omega,\\
&x_{\gamma}&&\text{if}\ \gamma\nless\omega
\end{aligned}
\right.
$$ 
for every $\gamma\in\Omega$. We denote by $\Psi$ the subgroup of~$\Aut(\F)$ generated by the elements~$\psi_{\omega}$, $\omega\in\Omega$.
Now, we consider the semi-direct product $\F\rtimes\Psi$ corresponding to the tautological action of~$\Psi\subset\Aut(\F)$ on~$\F$ and we consider the elements $s_{\omega}=x_{\omega}\psi_{\omega}\in \F\rtimes\Psi$. Denote by $S$ the set consisting of $m$ elements~$s_{\omega}$ and by~$W$ the subgroup of~$\F\rtimes\Psi$ generated by the elements~$s_{\omega}$. Obviously, $\psi_{\omega}^2=1$ and $s_{\omega}^2=1$ for every $\omega\in\Omega$.

\begin{propos}
Suppose~$\omega_1<\omega_2$; then $\psi_{\omega_1}\psi_{\omega_2}= \psi_{\omega_2}\psi_{\omega_1}$ and $s_{\omega_1}s_{\omega_2}=s_{\omega_2}s_{\omega_1}$.
\end{propos}

\begin{proof}
Suppose $\gamma\in\Omega$. By a direct computation, we obtain that 
$$
\psi_{\omega_1}(\psi_{\omega_2}(x_{\gamma}))= \psi_{\omega_2}(\psi_{\omega_1}(x_{\gamma}))=
\left\{
\begin{aligned}
x_{\omega_2}x_{\omega_1}&x_{\gamma} x_{\omega_1}x_{\omega_2}&&\text{if}\ \gamma<\omega_1,\\
x_{\omega_2}&x_{\gamma} x_{\omega_2}&&\text{if}\ \gamma\nless\omega_1\ \text{and}\  
\gamma<\omega_2,\\
&x_{\gamma}&&\text{if}\ \gamma\nless\omega_2.
\end{aligned}
\right.
$$
Hence, the automorphisms~$\psi_{\omega_1}$ and~$\psi_{\omega_2}$ commute. Now, we have
\begin{gather*}
s_{\omega_1}s_{\omega_2}= x_{\omega_1}\psi_{\omega_1} x_{\omega_2}\psi_{\omega_2}=
x_{\omega_1}x_{\omega_2}\psi_{\omega_1} \psi_{\omega_2},\\
s_{\omega_2}s_{\omega_1}= x_{\omega_2}\psi_{\omega_2} x_{\omega_1}\psi_{\omega_1}=
x_{\omega_2}(x_{\omega_2}x_{\omega_1}x_{\omega_2})
\psi_{\omega_2} \psi_{\omega_1}= x_{\omega_1}x_{\omega_2}\psi_{\omega_1} \psi_{\omega_2}.
\end{gather*}
Therefore, $s_{\omega_1}s_{\omega_2}=s_{\omega_2}s_{\omega_1}$.
\end{proof}

\begin{propos}\label{propos_Coxeter}
$(W,S)$ is a right-angular Coxeter system such that generators~$s_{\omega_1}$ and $s_{\omega_2}$ commute if and only if~$\omega_1$ and $\omega_2$ are comparable. 
\end{propos}

\begin{proof}
We have already proved that the elements~$s_{\omega}$ satisfy relations $s_{\omega}^2=1$ and $s_{\omega_1}s_{\omega_2}=s_{\omega_2}s_{\omega_1}$ whenever $\omega_1<\omega_2$. We need to prove that these relations imply all relations among the elements~$s_{\omega}$. Let $\underline{W}$ be the right-angular Coxeter group  generated by the elements~$r_{\omega}$, $\omega\in\Omega$, with all relations among them following from the relations $r_{\omega}^2=1$ and $r_{\omega_1}r_{\omega_2}=r_{\omega_2}r_{\omega_1}$, $\omega_1<\omega_2$. Then we have a well-defined homomorphism $i:\underline{W}\to \F\rtimes\Psi$ such that $i(r_{\omega})=s_{\omega}$. We need to prove that $i$ is injective.  Consider the homomorphism 
$j:\F\to \underline{W}$ given by $j(x_{\omega})=r_{\omega}$. Obviously, the homomorphism~$j$ is $\Psi$-invariant. Hence, there exists a well-defined homomorphism $j':\F\rtimes\Psi\to\underline{W}$ such that $j'(x\psi)=j(x)$ for every $x\in\F$ and every $\psi\in\Psi$. Since the composite homomorphism
$$
\underline{W}\xrightarrow{i}\F\rtimes\Psi \xrightarrow{j'}\underline{W}
$$
is the identity isomorphism, we obtain that $i$ is injective. 
\end{proof}

\begin{remark}
The above construction has the following generalization. Let $\mathbb{F}$ be the free group with generators~$x_{\omega}$, $\omega\in\Omega$. Consider the automorphisms $\psi_{\omega}\in\Aut(\mathbb{F})$ such that 
$$
\psi_{\omega}(x_{\gamma})=\left\{
\begin{aligned}
x_{\omega}&x_{\gamma}x_{\omega}^{-1}&&\text{if}\ \gamma<\omega,\\
&x_{\gamma}&&\text{if}\ \gamma\nless\omega
\end{aligned}
\right.
$$
for every $\gamma\in\Omega$, and consider the elements $y_{\omega}=x_{\omega}^{-1}\psi_{\omega}\in \mathbb{F}\rtimes\Aut(\mathbb{F})$. Then $y_{\omega_1}y_{\omega_2}=y_{\omega_2}y_{\omega_1}$ and these relations imply all relations among the elements~$y_{\omega}$. Therefore, the subgroup $A\subset \mathbb{F}\rtimes\Aut(\mathbb{F})$ generated by the elements~$y_{\omega}$ is the right-angular Artin group. The proof is similar to the proof of Proposition~\ref{propos_Coxeter}.
\end{remark}

For~$\omega\in\Omega$, let $\F_{>\omega}\subset \F$ be the subgroup generated by all elements $x_{\omega'}$ such that $\omega'>\omega$ and let $\F_{\ge\omega}\subset \F$ be the subgroup generated by all elements $x_{\omega'}$ such that $\omega'\ge\omega$. The definition of the automorphisms~$\psi_{\omega}$ immediately implies

\begin{propos}\label{lem_>omega}
The subgroups~$\F_{>\omega}$ and $\F_{\ge\omega}$ are $\Psi$-invariant. For each $\psi\in\Psi$, we have $\psi(x_{\omega})=yx_{\omega}y^{-1}$ for a $y\in \F_{>\omega}$. 
\end{propos}

Now, we consider the mapping $\F\rtimes\Psi\to \F$ such that $\psi x\mapsto x$ for every $x\in \F$ and every $\psi\in\Psi$, and we denote by~$\theta$ the restriction of this mapping to~$W$. The mapping~$\theta:W\to \F$ is not a homomorphism, but it satisfies the following property. 

\begin{propos}\label{propos_theta}
Suppose $g\in W$ and $\omega\in\Omega$; then there exists $y\in \F_{>\omega}$ such that
$$
\theta(s_{\omega}g)=yx_{\omega}y^{-1}\theta(g).
$$
\end{propos} 

\begin{proof}
Suppose $g=\psi x$, where $x\in \F$ and $\psi\in\Psi$; then $\theta(g)=x$. We have, 
$s_{\omega}g=\psi_{\omega}\psi zx$, where $z=\psi^{-1}(x_{\omega})$. Hence, $\theta(s_{\omega}g)=zx=z\theta(g)$. By Proposition~\ref{lem_>omega}, we have $z=yx_{\omega}y^{-1}$ for a $y\in \F_{>\omega}$. 
\end{proof}

Now, let $H$ be a subgroup of finite index in~$\F$. Denote by $\Psi_H$ the subgroup of~$\Psi$ consisting of all automorphisms~$\psi$ such that $\psi(x)H=xH$ for every $x\in\F$. In particular, $\psi(H)=H$ for every $\psi\in \Psi_H$, hence, the semi-direct product $H\rtimes \Psi_H$ is well defined.
It can be easily proved that the subgroup~$\Psi_H$ has finite index in~$\Psi$. We put, $W_H=W\cap(H\rtimes \Psi_H)$. Then $W_H$ is a subgroup of finite index in~$W$.

\begin{propos}\label{propos_thetaH}
Let $g_1$ and~$g_2$ be elements of~$W$ such that $g_1W_H=g_2W_H$. Then $\theta(g_1)H=\theta(g_2)H$. Consequently, the mapping~$\theta$ induces a well-defined mapping $\theta_H:W/W_H\to \F/H$. 
\end{propos}

\begin{proof}
We have, $g_2=g_1g$ for some $g\in H\rtimes \Psi_H$. Suppose, $g_1=\psi x$ and $g=\varphi h$, where $x\in \F$, $\psi\in\Psi$, $h\in H$, and $\varphi\in \Psi_H$. Then $\theta(g_1)=x$ and $\theta(g_2)=\varphi^{-1}(x)h$. Since $h\in H$ and $\varphi\in \Psi_H$, we see that $\theta(g_1)H=\theta(g_2)H$.
\end{proof}

\section{The manifold $M_{\Pi^n}$ is URC}
\label{section_Pi_constr}

In this section we use the group-theoretic construction described in section~\ref{section_group_constr} to prove Theorem~\ref{theorem_MPi}. Since arcwise connected components of~$X$ can be treated separately, we may assume that $X$ is  arcwise connected.

Recall that a simplicial complex~$Z$ is called an $n$-dimensional \textit{pseudo-manifold} if 
\begin{enumerate}
\item each simplex of~$Z$ is contained in an $n$-dimensional simplex;
\item each $(n-1)$-dimensional simplex of~$Z$ is contained in exactly two $n$-dimensional simplices.
\end{enumerate}
All pseudo-manifolds under consideration are supposed to be compact and \textit{strongly connected\/}. The latter means that any two $n$-dimensional simplices of~$Z$ can be connected by a sequence of $n$-dimensional simplices such that every two consecutive simplices have a common $(n-1)$-dimensional face.  

The following proposition follows easily from the definition of singular homology groups.

\begin{propos}
Let $X$ be an arcwise connected topological space and let $z\in H_n(X,\Z)$ be an arbitrary homology class. Then there exist a strongly connected oriented $n$-dimensional pseudo-manifold~$Z$ and a continuous mapping $f:Z\to X$ such that $f_*[Z]=z$.
\end{propos}

Hence, to prove that $M_{\Pi^n}$ is a URC-manifold we suffice to show that for each strongly connected oriented $n$-dimensional pseudo-manifold~$Z$, a multiple of the fundamental class of~$Z$ can be realised as an image of the fundamental class of a finite-sheeted covering of~$M_{\Pi^n}$. We shall construct explicitly a finite-sheeted covering~$\hM_{\Pi^n}$ of~$M_{\Pi^n}$ and a continuous mapping $\hM_{\Pi^n}\to Z$ of non-zero degree. 

Replacing~$Z$ with its barycentric subdivision, we may assume that vertices of~$Z$ admit a regular colouring in~$n+1$ colours $1,\ldots,n+1$. (This means that any two vertices of~$Z$ connected by an edge are of distinct colours.) For an $n$-dimensional simplex~$\sigma$ of~$Z$, the colouring of its vertices in colours $1,\ldots,n+1$ provides the orientation of~$\sigma$. We colour the simplex~$\sigma$ in either white or black colour depending on whether this orientation of~$\sigma$ coincides with the global orientation of~$Z$ or not. Obviously, any two simplices with a common facet have distinct colours. Thus we obtain a chess colouring of $n$-dimensional simplices of~$Z$. 

We denote by $A$ the set of all $n$-dimensional simplices of~$Z$ and by~$A_+$ and~$A_-$ the sets of all white and black $n$-dimensional simplices of~$Z$ respectively. For a simplex $\tau$ of~$Z$, we denote by $A(\tau)$ the set of all $n$-dimensional simplices~$\sigma$ containing~$\tau$. We put $A_+(\tau)=A(\tau)\cap A_+$ and~$A_-(\tau)=A(\tau)\cap A_-$. 

\begin{propos}
For each simplex~$\tau$ of~$Z$ such that $\dim\tau<n$, the number of white $n$-dimensional simplices $\sigma$ containing~$\tau$ is equal to the number of black $n$-dimensional simplices $\sigma$ containing~$\tau$, i.\,e., $|A_+(\tau)|=|A_-(\tau)|$. 
\end{propos}

\begin{proof}
Let $k=\dim\tau$. Each $n$-dimensional simplex~$\sigma\in A(\tau)$ contains exactly  $n-k$ simplices~$\rho$ such that $\dim\rho=n-1$ and $\rho\supset\tau$. On the other hand, every $(n-1)$-dimensional simplex~$\rho$ containing~$\tau$ is contained in exactly one simplex belonging to~$A_+(\tau)$ and in exactly one simplex belonging to~$A_-(\tau)$. Hence, the number of $(n-1)$-dimensional simplices~$\rho$ containing~$\tau$ is equal to both $(n-k)|A_+(\tau)|$ and $(n-k)|A_-(\tau)|$. Therefore, $|A_+(\tau)|=|A_-(\tau)|$.
\end{proof}

This proposition implies that we can pair off elements of the sets~$A_+(\tau)$ and~$A_-(\tau)$. For every~$\tau$, we choose arbitrarily such pairing~$\lambda_{\tau}$. Then $\lambda_{\tau}:A(\tau)\to A(\tau)$ is a bijection such that $\lambda_{\tau}(A_{\pm}(\tau))=A_{\mp}(\tau)$ and $\lambda_{\tau}^2=\mathrm{id}$. 

Consider a non-empty proper subset $\omega\subset[n+1]$. Denote by~$\T_{\omega}$ the set of all simplices~$\tau$ of~$Z$ such that
the set of colours of vertices of~$\tau$ coincides with~$\omega$. If $|\omega|=k$, then the set~$\T_{\omega}$ consists of $(k-1)$-dimensional simplices. Obviously, each simplex~$\sigma\in A$ contains exactly one simplex $\tau\in\T_{\omega}$, which will be called the \textit{face of~$\sigma$ of type~$\omega$\/}. Hence the set $A$ is the disjoint union of the sets~$A(\tau)$, $\tau\in\T_{\omega}$. Let $\lambda_{\omega}:A\to A$ be the bijection whose restriction to every set~$A(\tau)$ coincides with~$\lambda_{\tau}$.
Then $\lambda_{\omega}$ is a permutation of the set~$A$ satisfying the following conditions:
\begin{enumerate}
\item $\lambda_{\omega}(A_+)=A_-$ and $\lambda_{\omega}(A_-)=A_+$.
\item $\lambda_{\omega}^2=1$.
\item For each $\sigma\in A$, the simplices $\sigma$ and~$\lambda_{\omega}(\sigma)$ have common face of type~$\omega$.
\end{enumerate}

Let $\Delta^{n}\subset\R^{n+1}$ be the standard simplex with vertices 
$$e_1=(1,0,\ldots,0),\ e_2=(0,1,\ldots,0),\ \ldots,\  e_{n+1}=(0,0,\ldots,1).$$ For a non-empty proper subset $\omega\subset[n+1]$, we denote by $\Delta_{\omega}$ the face of~$\Delta^n$ with vertices~$e_i$ such that $i\in\omega$. For each simplex~$\sigma\in A$, let $\iota_{\sigma}:\Delta^n\to\sigma$ be the linear isomorphism taking every vertex~$e_i$ to the vertex of~$\sigma$ of colour~$i$. Obviously, $\iota_{\sigma}$ preserves the orientation if~$\sigma\in A_+$ and reverses the orientation if~$\sigma\in A_-$. We shall denote the point~$\iota_{\sigma}(p)$ by $[p,\sigma]$. Then
\begin{equation}\label{eq_lambda}
[p,\lambda_{\omega}(\sigma)]=[p,\sigma]
\end{equation}
whenever $p\in\Delta_{\omega}$.

Now, let $\Omega$ be the set of all non-empty proper subsets $\omega\subset[n+1]$ partially oredered by inclusion; then $m=|\Omega|=2^{n+1}-2$. We consider the construction in section~\ref{section_group_constr} for this~$\Omega$. 
Recall that facets~$F_{\omega}$ of~$\Pi^n$ are in one-to-one correspondence with subsets~$\omega\in\Omega$, and $F_{\omega_1}\cap F_{\omega_2}\ne\emptyset$ if and only if either $\omega_1\subset\omega_2$ or $\omega_2\subset\omega_1$. Hence the right-angular Coxeter group~$W$ corresponding to the partially ordered set~$\Omega$ coincides with the right-angular Coxeter group~$W_{\Pi^n}$.
We have $M_{\Pi^n}=\U_{\Pi^n}/\Gamma$, where $\Gamma$ is the kernel of the homomorphism $\rho:W\to\Z_2^n$ such that $\rho(s_{\omega})=b_{|\omega|}$, $\omega\in\Omega$. 

Let us construct a degree~$1$ mapping $\pi:\Pi^n\to\Delta^n$ such that $\pi(F_{\omega})\subset \Delta_{\omega}$ for every $\omega\in\Omega$. 
Map the barycentre of~$\Pi^n$ to the barycentre of~$\Delta^n$.
Each face $F$ of~$\Pi^n$ can be uniquely written as $F_{\omega_1}\cap\ldots\cap F_{\omega_k}$, $\omega_1\subset\cdots\subset\omega_k$. Map the barycentre of~$F$ to the barycentre of~$\Delta_{\omega_1}$.  Now the mapping~$\pi$ is defined on the vertices of the barycentric subdivision of~$\Pi^n$. Extend this mapping linearly to every simplex of the barycentric subdivision of~$\Pi^n$.

Denote by~$S_A$ the group of permutations of the set~$A$. Consider the homomorphism~$\Lambda:\F\to S_A$ given by $\Lambda(x_{\omega})=\lambda_{\omega}$. This homomorphism yields the action of the group~$\F$ on the set~$A$. Since $Z$ is strongly connected, this action is transitive. We shall write $g\cdot \sigma=\Lambda(g)(\sigma)$ for $g\in \F$, $\sigma\in A$. Choose a simplex~$\sigma_0\in A_+$ and define the mapping $f:\U_{\Pi^n}\to Z$ by
\begin{equation*}
f([p,g])=[\pi(p),\theta(g)\cdot\sigma_0].
\end{equation*}
Let $H\subset \F$ be the stabilizer of~$\sigma_0$. The subgroup~$H$ has finite index, since the set~$A$ is finite. Then~$W_H$ is a  subgroup of finite index in~$W$. Let $\Gamma_H=W_H\cap\Gamma$ and $\hM_{\Pi^n}=\U_{\Pi^n}/\Gamma_H$.

\begin{propos}
The mapping~$f$ is well defined and invariant under the right action of~$W_H$ on~$\U_{\Pi^n}$. Hence $f$ induces a well-defined mapping $f_1:\hM_{\Pi^n}\to Z$. The degree of~$f_1$ is equal to~$\frac{|W:\Gamma_H|}{|A|}$.
\end{propos}

\begin{proof}
To prove that $f$ is well defined we need to show that 
$$[\pi(p),\theta(g)\cdot\sigma_0]=[\pi(p),\theta(g')\cdot\sigma_0]$$ whenever $g'g^{-1}\in W(p)$.
Let $F_{\omega_1}\cap\ldots\cap F_{\omega_k}$ be the minimal face of~$\Pi^n$ that contains~$p$. We may renumerate the subsets $\omega_i$ so that $\omega_1\subset\cdots\subset\omega_k$. Then $\pi(p)\in\Delta_{\omega_1}$. 
The  subgroup $W(p)$ is generated by~$s_{\omega_1},\ldots,s_{\omega_k}$. Since $g'g^{-1}\in W(p)$, Proposition~\ref{propos_theta} implies that $\theta(g')=x\theta(g)$ for an $x\in \F_{\ge\omega_1}$. Then $\Lambda(x)$ belongs to the subgroup of~$S_A$ generated by all $\lambda_{\omega}$ such that $\omega\supseteq\omega_1$. Let $\sigma=\theta(g)\cdot\sigma_0$; then $\theta(g')\cdot\sigma_0=x\cdot\sigma$.
Equation~\eqref{eq_lambda} implies that $[\pi(p),x\cdot\sigma] =[\pi(p),\sigma]$. 
Therefore $f$ is well defined. Proposition~\ref{propos_thetaH} immediately yields  that $f$ is $W_H$-invariant.

Let $\Gamma_H=W_H\cap\Gamma$. The manifold~$\U_{\Pi^n}$ is glued of permutahedra indexed by elements $g\in W$.
We denote by $\Pi_g$ the permutahedron corresponding to~$g$. The restriction $f|_{\Pi_g}$ coincides with the composite mapping
$$
\Pi_{g}\xrightarrow{\iota_g^{-1}}\Pi^n \xrightarrow{\pi}\Delta^n \xrightarrow{\iota_{\theta(g)\cdot\sigma_0}} \theta(g)\cdot\sigma_0,
$$
where $\iota_g:\Pi^n\to\U_{\Pi^n}$ is the mapping given by $\iota_g(p)=[p,g]$. All mappings in this diagram have degrees~$\pm 1$. The mapping~$\pi$  has degree~$1$. The mapping~$\iota_g$ has degree~$1$ if and only if the element~$g$ is represented by a word of even length in generators~$s_{\omega}$. The mapping~$\iota_{\theta(g)\cdot\sigma_0}$ has degree~$1$ if and only if~$\theta(g)\cdot \sigma_0\in A_+$. Obviously, $\theta(g)\cdot \sigma_0\in A_+$ if and only if the element~$\theta(g)\in \F$ is represented by a word of even length in generators~$x_{\omega}$. Propostition~\ref{propos_theta} implies that the parity of the length of the word in~$x_{\omega}$ representing~$\theta(g)$ coincides with the parity of the length of the word in~$s_{\omega}$ representing~$g$. Therefore, the mapping $f|_{\Pi_{g}}:\Pi_{g}\to\theta(g)\cdot\sigma_0$ always has degree~$1$. Now we see that $\hM_{\Pi^n}$ is a connected oriented manifold glued out of~$|W:\Gamma_H|$ permutahedra and $Z$ is a strongly connected oriented pseudo-manifold glued out of~$|A|$ simplices. Each cell of~$\hM_{\Pi^n}$ is mapped by~$f_1$ onto a simplex of~$Z$ with degree~$1$. Hence, the degree of~$f_1$ is equal to $\frac{|W:\Gamma_H|}{|A|}$.
\end{proof}

\begin{remark}
By a theorem of Thom~\cite{Tho54}, for each~$n$, there is a positive integer~$k(n)$ such that the class~$k(n)z$ is realisable for every~$X$ and every $z\in H_n(X,\Z)$. Estimates for~$k(n)$ were obtained by Novikov~\cite{Nov62} and Buchstaber~\cite{Buc69}. The best known estimate is that $k(n)$ divides the number $\prod p^{\left[
\frac{n-2}{2(p-1)}
\right]}$, where the product is taken over all odd primes~$p$, see~\cite{Buc69}. Unfortunately, the combinatorial construction described above does not allow us to obtain a reasonable estimate for~$k(n)$. The multiplicity $\frac{|W:\Gamma_H|}{|A|}$ obtained by this construction can be huge. 
\end{remark}

\section{Proof of Theorem~\ref{theorem_sc}}
\label{section_proof_sc}

In this section we shall prove all assertions of Theorem~\ref{theorem_sc} except for~(e). It was shown in section~\ref{section_urc_sc} that Theorem~\ref{theorem_sc}(e) follows from Theorem~\ref{theorem_CAT(-1)}. The proof of Theorem~\ref{theorem_CAT(-1)} will be given in the next section. 

Since $\RZ_{\Pi^n}$ is a URC-manifold and $K_{\Pi^n}$ is isomorphic to~$(\partial\Delta^n)'$, the following proposition implies Theorem~\ref{theorem_sc}(a).

\begin{propos}\label{propos_simp_map}
Let $P_1$ and $P_2$ be $n$-dimensional simple cells such that there exists a simplicial mapping $\varphi:K_{P_1}\to K_{P_2}$ of non-zero degree. Then $\RZ_{P_1}\geqslant\RZ_{P_2}$.
\end{propos}

\begin{proof}
Let $F_1,\ldots,F_{m_1}$ be the facets of~$P_1$ and let $G_1,\ldots,G_{m_2}$ be the facets of~$P_2$. Let~$\bar\varphi:P_1\to P_2$ be the mapping induced by~$\varphi$. Identifying the vertex sets of~$K_{P_1}$ and~$K_{P_2}$ with the sets~$[m_1]$ and~$[m_2]$ respectively, we may regard~$\varphi$ as a mapping $[m_1]\to [m_2]$. Let $\mu:\Z_2^{m_1}\to\Z_2^{m_2}$ be the homomorphism given by $\mu(a_i)=b_{\varphi(i)}$, where $a_1,\ldots,a_{m_1}$ and $b_1,\ldots,b_{m_2}$ are the standard generators of the groups~$\Z_2^{m_1}$ and~$\Z_2^{m_2}$ respectively. Since $\deg\varphi\ne 0$, we easily see that $\mu$ is surjective. Now we consider the mapping $f:\RZ_{P_1}\to\RZ_{P_2}$ given by $f([p,g])=[\bar\varphi(p),\mu(g)]$. Since  $\bar\varphi(F_i)\subset G_{\varphi(i)}$ for every~$i$, we see that the mapping~$f$ is well defined. Obviously, $f$ maps each $n$-cell of~$\RZ_{P_1}$ onto an $n$-cell of~$\RZ_{P_2}$ with degree~$\deg\varphi$. Therefore, $\deg f=2^{m_2-m_1}\deg\varphi\ne 0.$
\end{proof}

To prove Theorems~\ref{theorem_sc}(b),\,(c),\,(d) we shall show that (b)$\Rightarrow$(a) and (d)$\Rightarrow$(c)$\Rightarrow$(a).

\begin{proof}[Proof of (b)$\Rightarrow$(a)]

For any $(n-1)$-dimensional combinatorial sphere~$L$, there exists a simplicial mapping  $\varphi:L\to\partial\Delta^n$ of degree~$1$. This mapping induces a required simplicial mapping $\varphi':K_P=L'\to(\partial\Delta^n)'$.
\end{proof}

Let $P$ be a flag simple cell and let $0<\varepsilon<\pi$. We say that $P$ is $\varepsilon$-\textit{sparse} if there exists a homeomorphism $h:\partial P\to\bS^{n-1}$ such that $\dist(h(F_1),h(F_2))\ge\varepsilon$ for any two non-intersecting facets~$F_1$ and~$F_2$ of~$P$, and $\diam(h(F))\le\pi-\varepsilon$ for every facet~$F$ of~$P$. (The latter condition is technical. It guarantees that the open $\frac{\varepsilon}{2}$-neighborhood of every set $h(F)$ does not contain antipodal points.)

\begin{propos}\label{propos_finesparse}
Let $P_1$ and $P_2$ be $n$-dimensional simple cells such that $P_1$ is $\varepsilon$-fine and $P_2$ is flag and $\varepsilon$-sparse. Then there exists a simplicial mapping $K_{P_1}\to K_{P_2}$ of non-zero degree.
\end{propos}

\begin{proof}
Let $F_1,\ldots,F_{m_1}$ be the facets of~$P_1$ and let $G_1,\ldots,G_{m_2}$ be the facets of~$P_2$. Then $K_1=K_{P_1}$ and $K_2=K_{P_2}$ are simplicial complexes on the vertex sets $\{u_1,\ldots,u_{m_1}\}$ and $\{v_1,\ldots,v_{m_2}\}$ respectively.
Let $f:K_1\to\bS^{n-1}$ be the mapping of non-zero degree such that $\diam(f(\sigma))<\varepsilon$ for every simplex~$\sigma$ of~$K_1$. Let $h:\partial P_2\to\bS^{n-1}$ be the homeomorphism such that $\dist(h(G_i),h(G_j))\ge\varepsilon$ whenever $G_i\cap G_j=\emptyset$, and $\diam(h(G_i))\le\pi-\varepsilon$ for every~$i$. We define a mapping $\varphi:[m_1]\to [m_2]$ in the following way. 
For every $i\in[m_1]$, consider the point~$h^{-1}(f(u_i))\in\partial P_2$, choose an arbitrary facet~$G_j$ containing this point, and put $\varphi(i)=j$. Now let us prove that the mapping~$\varphi$ induces a simplicial mapping  
$\Phi:K_1\to K_2$. For every $i$, we put $\Phi(u_i)=v_{\varphi(i)}$.
Suppose $[u_{i_1}\ldots u_{i_k}]$ is a simplex of~$K_1$. Then $\dist(f(u_{i_p}),f(u_{i_q}))<\varepsilon$ for any~$p$ and~$q$. But $f(u_{i_p})\in h(G_{\varphi(i_p)})$ and $f(u_{i_q})\in h(G_{\varphi(i_q)})$. Hence, $G_{\varphi(i_p)}\cap G_{\varphi(i_q)}\ne \emptyset$. Since $K_{2}$ is a flag complex, we obtain that $G_{\varphi(i_1)}\cap\ldots\cap G_{\varphi(i_k)}\ne\emptyset$. Therefore, the vertices~$\Phi(u_{i_1}),\ldots,\Phi(u_{i_k})$ span a simplex in~$K_{2}$. Thus, $\Phi$ uniquely extends to a simplicial mapping $\Phi:K_1\to K_2$.

To show that $\Phi$ has non-zero degree, we need to prove that the mappings $\Phi:K_1\to K_2$ and $h^{-1}f:K_1\to\partial P_2=K_2$ are homotopic. Equivalently, we need to prove that the mappings $h\Phi$ and $f$ of $K_1$ to $\bS^{n-1}$ are homotopic. Let $x\in K_1=\partial P_1$ be an arbitrary point. Let $F_i$ be a facet of $P_1$ containing~$x$ and let $\sigma$ be a simplex of~$K_1$ containing~$x$. Then $u_i\in\sigma$, hence, $\dist(f(x),f(u_i))<\varepsilon$. Since $\Phi(F_i)\subset G_{\varphi(i)}$, we obtain that $h(\Phi(x))\in  h(G_{\varphi(i)})$. On the other hand, we have $f(u_i)\in h(G_{\varphi(i)})$ by the definition of~$\varphi$. Hence, $\dist(h(\Phi(x)),f(u_i))\le\diam(h( G_{\varphi(i)}))\le\pi-\varepsilon$. Therefore, $\dist(h(\Phi(x)),f(x))<\pi$. Since this inequality holds for every~$x\in K_1$, we see that $h\Phi$ is homotopic to~$f$.
\end{proof}

\begin{cor}\label{cor_finesparse}
Let $P_1$ and $P_2$ be $n$-dimensional simple cells such that $P_1$ is  $\varepsilon$-fine and $P_2$ is flag and $\varepsilon$-sparse. Then $\RZ_{P_1}\geqslant\RZ_{P_2}$.
\end{cor}

\begin{propos}\label{propos_permutohedron}
For every $n\ge 2$, the permutahedron~$\Pi^n$ is $\varepsilon_n$-sparse.
\end{propos}

\begin{proof}
We consider the standard realisation of~$\Pi^n$ in~$\R^{n+1}$ whose vertices are obtained by permuting the coordinates of the point $(1,2,\ldots,n+1)$. Let $o=(\frac{n+2}{2},\ldots,\frac{n+2}{2})$ be the centre of~$\Pi^n$. It is easy to compute that the circumscribed sphere of~$\Pi^n$ has radius $R_n=\sqrt{\frac{n(n+1)(n+2)}{12}}$.
Identify $\bS^{n-1}$ with the unit sphere with centre at~$o$ in the hyperplane containing~$\Pi^n$,  and let $h:\partial\Pi^n\to\bS^{n-1}$ be the radial projection from~$o$. Suppose, $x$ and $y$ are points belonging to non-intersecting facets~$F_{\omega_1}$ and $F_{\omega_2}$ of~$\Pi^n$ and let $\xi=(\xi_1,\ldots,\xi_{n+1})$ and $\eta=(\eta_1,\ldots,\eta_{n+1})$ be the vectors $\overrightarrow{ox}$ and~$\overrightarrow{oy}$ respectively. Since $F_{\omega_1}\cap F_{\omega_2}=\emptyset$, we see that neither $\omega_1\subset\omega_2$ nor  $\omega_2\subset\omega_1$. Choose any $j\in\omega_1\setminus\omega_2$ and any $k\in\omega_2\setminus\omega_1$. Then $\xi_j-\xi_k\ge 1$ and $\eta_k-\eta_j\ge 1$. We have,
$$
\cos\dist(h(x),h(y))= \frac{(\xi,\eta)}{|\xi||\eta|}= \frac{\sum_{i\ne j,k}\xi_i\eta_i+\xi_j\eta_k+\xi_k\eta_j-(\xi_j-\xi_k)(\eta_k-\eta_j)}{|\xi||\eta|}\le 1-\frac{1}{R_n^2}
$$
Hence, $\dist(h(x),h(y))\ge\varepsilon_n$.

Now, we need to prove that $\diam(h(F_{\omega}))\le\pi-\varepsilon_n$ for every~$\omega$. Let $B$ be the circumscribed ball of~$\Pi^n$ and let $H_{\omega}$ be the affine plane spanned by~$F_{\omega}$. It is easy to see that the distance from~$o$ to~$H_{\omega}$ is at least~$r_n=\frac{1}{2}\sqrt{n(n+1)}$. (The value~$r_n$ is attained if $|\omega|$ is either~$1$ or~$n$.) Since $F_{\omega}\subset H_{\omega}\cap B$, we have
$$
\diam(h(F_{\omega}))\le 2\arccos\frac{r_n}{R_n}=\pi-\arccos\left(1-\frac{6}{n+2}\right)\le \varepsilon_n.
$$
Thus $\Pi^n$ is $\varepsilon_n$-sparse.
\end{proof}

Propositions~\ref{propos_finesparse} and~\ref{propos_permutohedron} imply that (c)$\Rightarrow$(a). Since $$\rho_n=\arcsinh\cot\frac{\varepsilon_n}{2},$$ the following proposition yields~(d)$\Rightarrow$(c). 

\begin{propos}\label{propos_rho}
Let $0<\varepsilon\le\frac{\pi}{2}$ and let $P\subset\Lob^n$ be a simple convex polytope whose interior contains a closed ball of radius $\rho=\arcsinh\cot\frac{\varepsilon}{2}$. Then $P$ is $\varepsilon$-fine.
\end{propos}

\begin{proof}
Let $B\subset P$ be a ball of radius~$\rho$ and let $o$ be its centre. Identify $\bS^{n-1}$ with the sphere of unit vectors in~$T_{o}\Lob^n$. Define the mapping $f:\Lob^n\setminus\{ o\}\to \bS^{n-1}$ by taking every point $x$ to the tangent vector to the line segment $[ox]$ at~$o$. 
Obviously, the degree of~$f|_{\partial P}$ is equal to~$1$.

Let $F_1,\ldots,F_m$ be the facets of~$P$ and let $v_1,\ldots,v_m$ be the corresponding vertices of~$K_P$. For each~$i$, let $p_i$ be the orthogonal projection of the point $o$ onto the hyperplane spanned by~$F_i$ and let $\xi_i=f(p_i)$. For each point $x\in F_i$, $op_ix$ is a triangle with right angle at the vertex~$p_i$ and angle $\dist_{\bS^{n-1}}(f(x),\xi_i)$ at the vertex~$o$. Hence,
$$
\cot\bigl(\dist_{\bS^{n-1}}(f(x),\xi_i)\bigr)=
\sinh\bigl(\dist_{\Lob^n}(o,p_i)\bigr)\coth\bigl(\dist_{\Lob^n}(x,p_i)\bigr)>\sinh\rho=\cot\frac{\varepsilon}{2}
$$
Therefore, 
\begin{equation}
\label{eq_-1ineq}
\dist_{\bS^{n-1}}(f(x),\xi_i)<\frac{\varepsilon}{2}
\end{equation}
Consequently, $\dist_{\bS^{n-1}}(\xi_i,\xi_j)<\varepsilon$ whenever $F_i\cap F_j\ne\emptyset$.

Define a mapping $\tilde{f}:K_P\to\bS^{n-1}$ by putting 
$$
\tilde{f}(x)=\frac{\beta_1\xi_{i_1}+\ldots+\beta_k\xi_{i_k}}{|\beta_1\xi_{i_1}+\ldots+\beta_k\xi_{i_k}|}
$$
for a point $x$ that has barycentric coordinates $\beta_1,\ldots,\beta_k$ in a simplex $[v_{i_1}\ldots v_{i_k}]$. Obviously, $\tilde{f}$ is well defined. We have~$\tilde{f}(v_i)=\xi_i$. For each simplex $\Delta=[v_{i_1}\ldots v_{i_k}]$ of~$K_P$, $\tilde{f}(\Delta)$ is the convex hull of the points~$\xi_{i_1},\ldots,\xi_{i_k}$ in~$\bS^{n-1}$. Pairwise distances between the points $\xi_{i_1},\ldots,\xi_{i_k}$ are less than~$\varepsilon$. Since $\varepsilon\le\frac{\pi}{2}$, this implies that $\diam(\tilde{f}(\Delta))<\varepsilon$. Besides, inequality~\eqref{eq_-1ineq} implies that the set $f(\Delta)$ is contained in an $\frac{\varepsilon}{2}$-neighbourhood of the convex simplex~$\tilde{f}(\Delta)$. It follows that $\dist_{\bS^{n-1}}(\tilde{f}(x),f(x))<\pi$ for every $x\in K_P=\partial P$. Hence $\tilde{f}$ is homotopic to $f|_{\partial P}$. Therefore, $\tilde{f}$ has degree~$1$.
Thus $P$ is $\varepsilon$-fine.
\end{proof}

\section{Proofs of Theorems~\ref{theorem_hyp} 
and~\ref{theorem_CAT(-1)}}
\label{section_proof_hyp}

We have shown in section~\ref{section_urc_sc} that Theorem~\ref{theorem_CAT(-1)} implies Theorem~\ref{theorem_hyp}. Nevertheless, it is useful to give a proof of Theorem~\ref{theorem_hyp} first, because this proof contains all ideas behind the proof 
of Theorem~\ref{theorem_CAT(-1)}.

\begin{proof}[Proof of Theorem~\ref{theorem_hyp}]

Let $\M$ be the set of all mirrors of reflections $s\in W$. Choose a point $o\in \Lob^n$ and denote by~$\M_{>\rho_n}$ the set of all mirrors~$H\in\M$ such that the distance from~$o$ to~$H$ is greater than~$\rho_n$. For each~$H\in\M_{>\rho_n}$, let $\Lambda^+_H\subset\Lob^n$ be the closed half-space bounded by~$H$ and containing~$o$. Denote by~$P_1$ the intersection of all half-spaces~$\Lambda^+_H$, $H\in\M_{>\rho_n}$, see Figure~\ref{fig_P1}. (In this figure  the central pentagon is~$P$, the dashed circle is the sphere of centre~$o$ and radius~$\rho_n$, and the polygon bounded by the thick line is  $P_1$.) Obviously,  $P_1$ is a compact convex polytope. Since mirrors in~$\M$ are orthogonal whenever intersect, we see that the polytope~$P_1$ is right-angular. In particular, it is simple. Let $W_1$ be the group generated by the reflections in facets of~$P_1$. Then $W_1$ is a subgroup of finite index in~$W$. Let $F_1,\ldots,F_{m_1}$ be the facets of~$P_1$ and let $s_1,\ldots,s_{m_1}$ be the orthogonal reflections in these facets respectively. Consider the homomorphism $\eta_1:W_1\to\Z_2^{m_1}$ that takes each~$s_i$ to the generator of the $i$th factor~$\Z_2$ and put $\Gamma_1=\ker\eta_1$. Then $\Gamma_1$ acts freely on~$\Lob^n$ and the quotient~$\Lob^n/\Gamma_1$ is homeomorphic to~$\RZ_{P_1}$. Since the interior of~$P_1$ contains a ball of radius~$\rho_n$, we obtain that $\Lob^n/\Gamma_1$ is a URC-manifold. But the manifolds~$\Lob^n/\Gamma$ and $\Lob^n/\Gamma_1$ possess a common finite-sheeted covering, since the groups~$\Gamma$ and~$\Gamma_1$ are commensurable. Therefore, $\Lob^n/\Gamma$ is a URC-manifold.
\end{proof}

\begin{figure}
\begin{center}
\unitlength=1cm
\includegraphics[scale=.8]{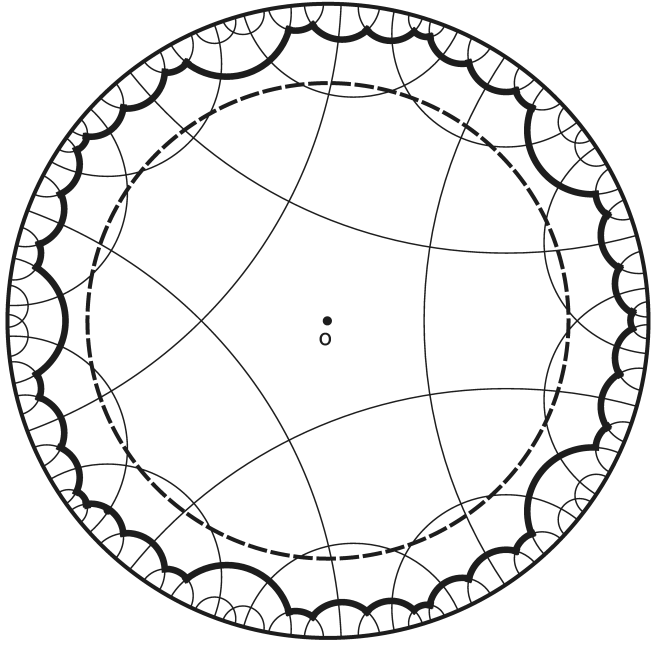}
\end{center}
\caption{Polytope~$P_1$}\label{fig_P1}
\end{figure}

\begin{proof}[Proof of Theorem~\ref{theorem_CAT(-1)}]
The idea is to mimic the above proofs of Proposition~\ref{propos_rho} and Theorem~\ref{theorem_hyp} using the  $\CAT(-1)$ inequality instead of constant negative curvature to obtain an analogue of inequality~\eqref{eq_-1ineq}. However, we shall face some additional difficulties on this way.

The group~$W=W_P$ acts by reflections on the $\CAT(-1)$ mani\-fold~$\U_P$. Let $\M$ be the set of all mirrors of reflections $s\in W$. Let $o\in \U_P$ be a point with a neighbourhood isometric to an open subset of~$\Lob^n$. Then the tangent space~$T_o\U_P$ is well defined. Let $\rho$ be a sufficiently large positive number. (A particular $\rho$ will be chosen later.) Denote by~$\M_{>\rho}$ the set of all mirrors~$H\in\M$ such that the distance from~$o$ to~$H$ is greater than~$\rho$. For each $H\in\M_{>\rho}$, let $\Lambda^+_H\subset\U_P$ be the closed half-space bounded by~$H$ and containing~$o$. Denote by~$P_1$ the intersection of all half-spaces~$\Lambda^+_H$, $H\in\M_{>\rho}$. Let $H_1,\ldots,H_{m_1}\in\M_{>\rho}$ be all mirrors intersecting~$P_1$; then the subsets~$F_i=H_i\cap P_1$ will be called \textit{facets\/} of~$P_1$ and their non-empty intersections will be called \textit{faces\/} of~$P_1$. (Since $W$ is a \textit{right-angular\/} Coxeter group, it follows that every~$F_i$ is $(n-1)$-dimensional.) Let $W_1\subset W$ be the subgroup of finite index  generated by the reflections in mirrors $H_1,\ldots,H_{m_1}$. 

If $P_1$ were a simple cell, $\RZ_{P_1}$ would be a finite-sheeted covering of~$\RZ_P$. Hence we would suffice to prove that~$\RZ_{P_1}$ is a URC-manifold. The author does not know whether $P_1$ is always a simple cell (see Remark~\ref{remark_simple_cell} below). Nevertheless, $P_1$ ``looks like'' a simple cell so that we can mimic the proofs of Propositions~\ref{propos_simp_map}, \ref{propos_finesparse}, and~\ref{propos_rho} in this context. The main difficulty is that we cannot identify $P_1$ with a cone over the dual combinatorial sphere~$K_{P_1}$ and, hence, should avoid usage of~$K_{P_1}$.  We start with the following obvious 

\begin{propos}
There exists a mapping $t:\Pi^n\to\Pi^n$ such that $t(U_{\omega})\subset F_{\omega}$ for some open neighbourhoods~$U_{\omega}$ of  facets~$F_{\omega}\subset\Pi^n$.
\end{propos}
 
Notice that we denote facets of~$P_1$ by $F_i$ and facets of~$\Pi^n$ by~$F_{\omega}$. This will not lead to a confusion.

Obviously, the mapping~$t$ has degree~$1$. 
Now we consider the mapping $$\varphi=th^{-1}:\bS^{n-1}\to\partial\Pi^n,$$ where $h:\partial\Pi^n\to \bS^{n-1}$ is the radial projection constructed in the proof of Proposition~\ref{propos_permutohedron}. Since the permutahedron has finitely many facets and every facet is compact, we obtain that the mapping~$\varphi$ satisfies the following property:

\begin{itemize}
\item[$(*)$]\textit{There is an~$\varepsilon>0$ such that $\varphi(x)\in F_{\omega}$ whenever $\dist(x,h(F_{\omega}))\le\varepsilon$.}  
\end{itemize}

Now we are ready to choose a particular~$\rho>0$. We take $\rho=\arcsinh\cot\frac{\varepsilon}{2}$.

Identify~$\bS^{n-1}$ with the unit sphere in~$T_o\U_P$. Since $\U_P$ is a $\CAT(-1)$ manifold,  the point~$o$ is joined with every point~$x\in\U_P$ by a unique geodesic segment. We denote by~$f(x)$ the unit tangent vector to this geodesic segment at~$o$. Then $f:\U_P\setminus\{o\}\to\bS^{n-1}$ is a continuous mapping.

\begin{propos}\label{propos_estim}
For each facet $F_i$ of~$P_1$, the diameter of~$f(F_i)$ does not exceed~$\varepsilon$. 
\end{propos}

\begin{proof}
Let $H_i\in\M_{>\rho}$ be the mirror such that $H_i\cap P_1=F_i$. Then the distance~$d$ from~$o$ to~$H_i$ is greater than~$\rho$. Since $H_i$ is closed and the balls in~$\U_P$ are compact, we obtain that there exists a point~$p\in H_i$ such that $\dist(o,p)=d$.
Let $x$ be a point in~$F_i$. Consider the comparison triangle $o^*x^*p^*$ in~$\Lob^2$ for the triangle~$oxp$. The geodesic segment~$[xp]$ is contained in~$H_i$. Hence, the distance from $o$ to every point $y$ of~$[xp]$ is greater than or equal to~$d$. By $\CAT(-1)$ inequality, the distance from~$o^*$ to every point of the segment~$[x^*p^*]$ is greater than or equal to~$d$. But the length of the segment~$[o^*p^*]$ is equal to~$d$. Hence the angle~$\beta^*$ of the triangle~$o^*x^*p^*$ at the vertex~$p^*$  is greater than or equal to~$\frac{\pi}{2}$. Let $\alpha^*$ be the angle of the triangle~$o^*x^*p^*$ at the vertex~$o^*$. Then 
$$
\cot\alpha^*\ge\sinh d>\sinh\rho=\cot\frac{\varepsilon}{2}.
$$
Therefore, $\alpha^*<\frac{\varepsilon}{2}$. 

A neighbourhood of~$o$ in~$\U_P$ is isometric to an open subset of~$\Lob^n$. Hence the angle~$\alpha$ of the triangle~$oxp$ at the vertex~$o$ is well defined and is equal to $\dist(f(x),f(p))$. The $\CAT(-1)$ inequality easily implies that 
$$
\dist(f(x),f(p))=\alpha\le\alpha^*<\frac{\varepsilon}{2}.
$$
Hence, for any two points $x,y\in F_i$, we have $\dist(f(x),f(y))<\varepsilon$. 
\end{proof}

For each~$i\in[m_1]$, choose a subset~$\omega_i$ such that $f(F_i)\cap h(F_{\omega_i})\ne \emptyset$.  Since $\diam(f(F_i))<\varepsilon$, property~$(*)$ implies that $\varphi(f(F_i))\subset F_{\omega_i}$. 
Define the homomorphism
$
\xi:\Z_2^{m_1}\to\Z_2^n
$
by $\xi(a_i)=b_{|\omega_i|}$, where $a_1,\ldots,a_{m_1}$ are the generators of~$\Z_2^{m_1}$  and $b_1,\ldots,b_n$ are the generators of~$\Z_2^n$.

Though we are not sure that $P_1$ is a simple cell, we can apply the construction of~$\RZ_{P_1}$ described in section~\ref{section_sc}. Namely, we put
$$
\RZ_{P_1}=(P_1\times \Z_2^{m_1})/\sim,
$$
where $(p,g)\sim(p',g')$ if and only if $p=p'$ and $g'g^{-1}$ belongs to the subgroup of~$\Z_2^{m_1}$ generated by all~$a_i$ such that $p\in F_i$. 
Let $\eta_1:W_1\to\Z_2^{m_1}$ be the homomorphism that takes the reflection in every mirror~$H_i$ to $a_i$, and let $\Gamma_1=\ker\eta_1$.
Then $\Gamma_1$ is torsion-free and $\RZ_{P_1}=\U_P/\Gamma_1$. Therefore $\RZ_{P_1}$ is a PL manifold.

Now we define a mapping $\Phi:\RZ_{P_1}\to M_{\Pi^n}$ by 
$$
\Phi([p,g])=[\varphi(f(p)),\xi(g)].
$$
Since $\varphi(f(F_i))\subset F_{\omega_i}$, we obtain that $\Phi$ is well defined. Since $f$ and~$\varphi$ are degree~$1$ mappings, we obtain that the degree of~$\Phi$ is equal to $2^{m_1-n}\ne 0$. Therefore, $\RZ_{P_1}$ is a URC-manifold. Since $\RZ_P$ and $\RZ_{P_1}$ possess a common finite-sheeted covering, we obtain that $\RZ_P$ is a URC-manifold.
\end{proof}

\begin{remark}
\label{remark_simple_cell} The set~$P_1$ is a convex subset of a $\CAT(-1)$ manifold, hence, $P_1$ is contractible. Similarly, faces of~$P_1$ are convex and, hence, contractible. Stone~\cite{Sto76} proved that closed metric balls in a piecewise Euclidean or piecewise hyperbolic $\CAT(0)$-manifold are homeomorphic to the standard ball, see also~\cite{DaJa91b} and~\cite{Dav02}. This proof seems to be possible to extend to the case of an arbitrary convex subset. However, to guarantee that $P_1$ is a simple cell we still need a stronger result. Namely, we need to prove that the homeomorphisms of facets of~$P_1$ onto the standard balls can be chosen piecewise linear and compatible to each other. The author does not know whether this is possible or not. Actually, even if it were possible to deduce that $P_1$ is a simple cell from results of~\cite{Sto76}, \cite{DaJa91b}, and~\cite{Dav02}, we would prefer to avoid usage of these results, since they are based on hard theorems of geometric topology.
\end{remark}


\begin{thebibliography}{99}

\bibitem{BrHa99} M. Bridson, A. Haefliger, \textit{Metric Spaces of Nonpositive Curvature\/}, Springer, New York, 1999.

\bibitem{Bro85} R. Brooks, \textit{On branched coverings of 3-manifolds which fiber over the circle\/}, J. Reine
Angew. Math. \textbf{362} (1985), 87--101.

\bibitem{Buc69} V. M. Buchstaber, \textit{Modules of differentials of the
Atiyah-Hirzebruch spectral sequence,~I,~II\/}, Matem. Sb. \textbf{78(120)}:2, 307--320; \textbf{83(125)}:1(9), 61--76; English transl., Math. USSR Sb. \textbf{7}:2 (1969),
299--313; \textbf{12}:1 (1970), 59--75.

\bibitem{BuPa02} V. M. Buchstaber, T. E. Panov, {\it Torus
actions and their applications in topology and combinatorics}, University Lecture Series,
no.~24, Amer. Math. Soc., Providence, RI, 2002.


\bibitem{CaTo89} J. A. Carlson, D. Toledo, \textit{Harmonic mappings of K\"ahler manifolds to locally symmetric spaces\/}, Publ. Math. I.H.E.S. \textbf{69} (1989), 173--201.

\bibitem{ChDa95} R. M. Charney, M. W. Davis, \textit{Strict hyperbolization\/}, Topology \textbf{34} (1995), 329--350.


\bibitem{Dav83} M. W. Davis, \textit{Groups generated by reflections and aspherical manifolds not covered by Euclidean space\/}, Ann. Math. (2) \textbf{117}:2 (1983), 293--324.

\bibitem{Dav87} M. W. Davis, \textit{Some aspherical manifolds\/}, Duke Math. J. \textbf{55}:1 (1987), 105--139.

\bibitem{Dav02} M. W. Davis, \textit{Nonpositive curvature and reflection groups\/}, in \textit{Handbook of Geometric Topology\/} (ed. by R.~J.~Daverman and R.~B.~Sher), Elsevier Science B.V., 2002, 373--422.

\bibitem{DaJa91a} M. W. Davis, T. Januszkiewicz, \textit{Convex
polytopes, Coxeter orbifolds and torus actions\/},  Duke Math. J. \textbf{62}:2 (1991), 417--451.


\bibitem{DaJa91b} M. W. Davis, T. Januszkiewicz, \textit{Hyperbolization of polyhedra\/}, J. Diff. Geom. \textbf{34} (1991), 347--388.

\bibitem{Gai08a} A. A. Gaifullin, \textit{Realisation of cycles by aspherical manifolds}, Uspekhi Mat. Nauk \textbf{63}:3 (2008), 157--158; English transl., Russ. Math. Surv. \textbf{63}:3 (2008), 562--564; arXiv:0806.3580.

\bibitem{Gai08b} A. A. Gaifullin, \textit{The manifold of isospectral symmetric tridiagonal matrices and realization of cycles by aspherical manifolds\/}, Trudy MIAN \textbf{263} (2008), 44--63, English transl., Proc. Steklov Inst. Math. \textbf{263} (2008), 38--56.

\bibitem{Gai08c} A. A. Gaifullin, \textit{The construction of combinatorial manifolds with prescribed sets of links of vertices\/},  Izvestiya RAN, Ser. matem. \textbf{72}:5 (2008), 3--62; English transl., 
Izvestiya: Mathematics \textbf{72}:5 (2008), 845--899, arXiv:0801.4741.


\bibitem{Gro82} M. Gromov, \textit{Volume and bounded cohomology}, Publ. Math. I.H.E.S. \textbf{56} (1982), 5--99.

\bibitem{Gro87} M. Gromov, \textit{Hyperbolic groups\/}, Essays in Group Theory \textbf{8}, edited by S. M. Gersten,
Math. Sci. Res. Inst. Publ., Springer-Verlag, New York, 1987, 75--264.


\bibitem{KoLo08} D. Kotschick, C. L\"oh, \textit{Fundamental classes not representable by products\/}, J. London Math. Soc. \textbf{79}:3 (2009), 545--561; arXiv: 0806.4540. 

\bibitem{Loe31} F. L\"obell, \textit{Beispiele geschlossener dreidimensionaler Clifford--Kleinscher R\"aume negativer Kr\"um\-mung\/}, Ber. Sachs. Acad. Wiss. \textbf{83} (1931), 168--174.

\bibitem{MiTh77} J. W. Milnor, W. P. Thurston, \textit{Characteristic numbers of $3$-manifolds\/}, Enseign. Math. \textbf{23} (1977), 249--254.


\bibitem{Nov62} S. P. Novikov, \textit{Homotopy properties of Thom complexes\/}, Matem. Sb. \textbf{57(99)}:4 (1962), 407--442. 

\bibitem{Ont11} P. Ontaneda, \textit{Pinched Smooth Hyperbolization\/}, arXiv:1110.6374.

\bibitem{Sto76} D. Stone, \textit{Geodesics in piecewise linear manifolds\/}, Trans. Amer. Math. Soc. \textbf{215} (1976), 1--44.

\bibitem{Tho54} R. Thom, \textit{Quelques proprietes globales
des varietes differentiables}, Comm. Math. Helv.
\textbf{28} (1954), 17--86.


\bibitem{Tom84} C. Tomei, \textit{The topology of the isospectral
manifold of tridiagonal matrices\/}, Duke Math. J. \textbf{51}:4 (1984), 981--996.

\bibitem{Vin71} E. B. Vinberg, \textit{Discrete linear groups generated by reflections\/}, Izvestiya AN SSSR, Ser. Matem. \textbf{35}:5 (1971), 1072--1112; English transl., Math. USSR Izvestiya \textbf{5}:5 (1971), 1083--1119.

\bibitem{Vin84} E. B. Vinberg, \textit{Absence of crystallographic reflection groups in Lobachevsky space of large dimension\/}, Trudy Moskov. Mat. Obshch. \textbf{47} (1984), 68--102; English transl.,  Trans. Mosc. Math. Soc. \textbf{47} (1985), 75--112.


\bibitem{Wan02} S. Wang, \textit{Non-zero degree maps between $3$-manifolds\/}, Proc. of the ICM Beijing 2002, Vol.~\textbf{II}, 457--468, Higher Education Press Beijing, 2002.



\end{thebibliography}
\end{document}